\documentclass[letterpaper, 10 pt, conference]{ieeeconf}
\IEEEoverridecommandlockouts
\let\labelindent\relax
\usepackage{enumitem}
\usepackage{balance}
\usepackage[font={small}]{caption}
\usepackage{subcaption}
\usepackage{array}
\usepackage{textcomp}
\usepackage{mathtools, nccmath}
\usepackage{graphicx}
\usepackage{amsfonts}
\usepackage{amsmath}
\usepackage{amssymb}
\usepackage{algorithm}
\usepackage{algorithmic}
\usepackage{hyperref}
\usepackage{tikz}
\usepackage{arydshln}
\usepackage{multirow}
\usepackage{bm}
\usepackage{epstopdf}
\usepackage{cite}
\usepackage{siunitx}

\usepackage{amsthm}

\newtheorem{theorem}{Theorem}
\newtheorem{lemma}{Lemma}

\newtheorem{remark}{Remark}
\theoremstyle{definition}
\newtheorem{definition}{Definition}
\newtheorem{problem}{Problem}
\newtheorem{assumption}{Assumption}

\title{\LARGE \bf
Feasibility-Guaranteed Safety-Critical Control with Applications to\\ Heterogeneous Platoons}

\author{Shuo Liu$^{1}$, Wei Xiao$^{2}$ and Calin Belta$^{3}$
\thanks{This work was supported in part by the NSF under grant IIS-2024606.}
\thanks{$^{1}$S. Liu is with the department of Mechanical Engineering, Boston
University, Brookline, MA, 02215, USA. 
        {\tt\small \{liushuo\}@bu.edu}}%
\thanks{$^{2}$W. Xiao is with the Computer Science and Artificial Intelligence Lab, Massachusetts Institute of Technology, Cambridge, MA, USA 
        {\tt\small weixy@mit.edu}}%
\thanks{$^{3}$C. Belta is with the Department of Electrical and Computer Engineering and the Department of Computer Science, University of Maryland, College Park, MD, USA  
        {\tt\small calin@umd.edu}}
}

\begin{document} 
\maketitle

\begin{abstract}
This paper studies safety and feasibility guarantees for systems with tight control bounds. It has been shown that stabilizing an affine control system while optimizing a quadratic cost and satisfying state and control constraints can be mapped to a sequence of Quadratic Programs (QPs) using Control Barrier Functions (CBF) and Control Lyapunov Functions (CLF). One of the main challenges in this method is that the QP could easily become infeasible under safety constraints of high relative degree, especially under tight control bounds. Recent work focused on deriving sufficient conditions for guaranteeing feasibility. The existing results are case-dependent. In this paper, we consider the general case. We define a feasibility constraint and propose a new type of CBF to enforce it.  Our method guarantees the feasibility of the above mentioned QPs, while satisfying safety requirements. We demonstrate the proposed method on an Adaptive Cruise Control (ACC) problem for a heterogeneous  platoon with tight control bounds, and compare our method to existing CBF-CLF approaches. The results show that our proposed approach can generate gradually transitioned control (without abrupt changes) with guaranteed feasibility and safety. 
\end{abstract}

\section{Introduction}
\label{sec:Introduction}

Safety is the primary concern in
the design and operation of autonomous systems. Many existing works enforce safety as constraints in optimal control problems using Barrier Functions (BF) and Control Barrier Functions (CBF). BFs are Lyapunov-like functions \cite{tee2009barrier} whose use can be traced back to optimization problems \cite{boyd2004convex}. They have been utilized to prove set invariance \cite{aubin2011viability}, \cite{prajna2007framework} to derive multi-objective control \cite{panagou2013multi}, \cite{wang2016multi}, and to control multi-robot systems
\cite{glotfelter2017nonsmooth}. 

CBFs are extensions of BFs used to enforce safety, i.e., rendering a set forward invariant, for an affine control system. It was proved in \cite{ames2016control} that if a CBF for a safe set satisfies Lyapunov-like conditions, then this set is forward invariant and safety is guaranteed. It has also been shown that stabilizing an affine control system to admissible states, while minimizing a quadratic cost subject to state and control constraints, can be mapped to a sequence of Quadratic Programs (QPs) \cite{ames2016control} by unifying CBFs and Control Lyapunov Functions (CLFs) \cite{ames2012control}. In its original form, this approach, which in this paper we will refer to as CBF-CLF, works only for safety constraints with relative degree one. Exponential CBFs \cite{nguyen2016exponential} were introduced to accommodate higher relative degrees. A more general form of exponential CBFs, called High-Order CBFs (HOCBFs), has been proposed in \cite{xiao2021high}. The CBF-CLF method has been widely used to enforce safety in many applications, including rehabilitative system control \cite{isaly2020zeroing}, adaptive cruise control \cite{ames2016control}, humanoid robot walking \cite{khazoom2022humanoid} and robot swarming \cite{cavorsi2022multi}. However, the aforementioned CBF-CLF-QP might be infeasible in the presence of tight or time-varying control bounds due to the conflicts between CBF constraints and control bounds.

There are several approaches that aim to enhance the feasibility of the CBF-CLF method, while guaranteeing safety and satisfying the control bounds. The authors of \cite{zeng2021safety2} formulated CBFs as constraints in a Nonlinear Model Predictive Control (NMPC) framework, which allows the controller to predict future state information up to a horizon larger than one. This leads to a less aggressive strategy than the original
one-step ahead approach. However, the corresponding optimization is overall nonlinear and non-convex, and the computation is expensive. An iterative approach based on a convex MPC with linearized, discrete-time CBFs was proposed in \cite{liu2023iterative}. However,  the linearization affects the safety guarantee and global optimality. The authors of \cite{dawson2022safe} and \cite{du2023reinforcement} developed a model-based learning and model-free learning approach, respectively, to synthesize controllers with safety guarantees. However, integrating CBF-CLF with learning models introduces uncertainty about whether these safety guarantees can always be maintained. The works in  \cite{gurriet2018online,singletary2019online,gurriet2020scalable,chen2021backup} are based on a set of backup policies that are used to extend the safe set to a larger viable set to enhance the feasible space of the system in a finite time horizon under input constraints. This backup approach has further been generalized to infinite time horizons \cite{squires2018constructive}, \cite{breeden2021high}. One limitation of these approaches is that they require prior knowledge on backup sets, policies, or nominal control laws. Another limitation is that these approaches may introduce overly aggressive or conservative control strategies. 

Adaptive CBFs (aCBFs) \cite{xiao2021adaptive} have been proposed for time-varying control bounds by introducing penalty functions in the HOCBFs constraints. These  provide flexible and adaptive control strategies over time. An Auxiliary-Variable Adaptive CBF (AVCBF) method was proposed in \cite{liu2023auxiliary}, which preserves the adaptive property of aCBFs \cite{xiao2021adaptive}, while generating smooth control policies near the boundaries of safe sets. The smooth control policies help to regulate a system's behavior with gradually transitioned control and output. AVCBFs also require less additional constraints and simpler parameter tuning compared to aCBFs \cite{xiao2021adaptive}. These approaches, however, cannot guarantee the feasibility of the optimization by only making some hard constraints soft or by extending the feasible spaces of the safe sets. The work in \cite{xiao2022sufficient} provided sufficient conditions to guarantee the feasibility of the CBF-CLF-QPs without softening hard constraints. However, the method was developed for a particular case, and it is not clear how it can be generalized. 

In this paper, we generalize the method from \cite{xiao2022sufficient} by proposing a new type of CBF for safety-critical control problems. Specifically, we define a feasibility constraint, an auxiliary variable,  and a CBF-based equation for the auxiliary variable that works for general affine control systems. We guarantee feasibility and safety under tight control bounds. Moreover, the generated control policy is smooth, without abrupt changes. We demonstrate the effectiveness of the proposed method on an adaptive cruise control problem with tight control bounds, and compare it to the existing CBF-CLF approaches. The results show that our proposed approach can generate smoother control with guaranteed feasibility and safety. 

\section{Preliminaries}
\label{sec:Preliminaries}

Consider an affine control system of the form
\begin{equation}
\label{eq:affine-control-system}
\dot{\boldsymbol{x}}=f(\boldsymbol{x})+g(\boldsymbol{x})\boldsymbol{u},
\end{equation}
 where $\boldsymbol{x}\in \mathbb{R}^{n}, f:\mathbb{R}^{n}\to\mathbb{R}^{n}$ and $g:\mathbb{R}^{n}\to\mathbb{R}^{n\times q}$ are locally Lipschitz, and $\boldsymbol{u}\in \mathcal U\subset \mathbb{R}^{q}$, where $\mathcal U$ denotes the control limitation set, which is assumed to be in the form: 
\begin{equation}
\label{eq:control-constraint}
\mathcal U \coloneqq \{\boldsymbol{u}\in \mathbb{R}^{q}:\boldsymbol{u}_{min}\le \boldsymbol{u} \le \boldsymbol{u}_{max} \}, 
\end{equation}
with $\boldsymbol{u}_{min},\boldsymbol{u}_{max}\in \mathbb{R}^{q}$ (vector inequalities are interpreted componentwise).
 
\begin{definition}[Class $\kappa$ function~\cite{Khalil:1173048}]
\label{def:class-k-f}
A continuous function $\alpha:[0,a)\to[0,+\infty],a>0$ is called a class $\kappa$ function if it is strictly increasing and $\alpha(0)=0.$
\end{definition}

\begin{definition}
\label{def:forward-inv}
A set $\mathcal C\subset \mathbb{R}^{n}$ is forward invariant for system \eqref{eq:affine-control-system} if its solutions for some $\boldsymbol{u} \in \mathcal U$ starting from any $\boldsymbol{x}(0) \in \mathcal C$ satisfy $\boldsymbol{x}(t) \in \mathcal C, \forall t \ge 0.$
\end{definition}

\begin{definition}
\label{def:relative-degree}
The relative degree of a differentiable function $b:\mathbb{R}^{n}\to\mathbb{R}$ is the minimum number of times we need to differentiate it along dynamics \eqref{eq:affine-control-system} until any component of $\boldsymbol{u}$ explicitly shows in the corresponding derivative. 
\end{definition}

In this paper,  a \textbf{safety requirement} is defined as $b(\boldsymbol{x})\ge0$, and \textbf{safety} is the forward invariance of the set $\mathcal C\coloneqq \{\boldsymbol{x}\in\mathbb{R}^{n}:b(\boldsymbol{x})\ge 0\}$. The relative degree of function $b$ is also referred to as the relative degree of safety requirement $b(\boldsymbol{x}) \ge 0$. For a requirement $b(\boldsymbol{x})\ge0$ with relative degree $m$ and $\psi_{0}(\boldsymbol{x})\coloneqq b(\boldsymbol{x}),$ we define a sequence of functions $\psi_{i}:\mathbb{R}^{n}\to\mathbb{R},\ i\in \{1,...,m\}$ as

\begin{equation}
\label{eq:sequence-f1}
\psi_{i}(\boldsymbol{x})\coloneqq\dot{\psi}_{i-1}(\boldsymbol{x})+\alpha_{i}(\psi_{i-1}(\boldsymbol{x})),\ i\in \{1,...,m\}, 
\end{equation}
where $\alpha_{i}(\cdot ),\ i\in \{1,...,m\}$ denotes a $(m-i)^{th}$ order differentiable class $\kappa$ function. We further define a sequence of sets $\mathcal C_{i}$ based on \eqref{eq:sequence-f1} as
\begin{equation}
\label{eq:sequence-set1}
\mathcal C_{i}\coloneqq \{\boldsymbol{x}\in\mathbb{R}^{n}:\psi_{i}(\boldsymbol{x})\ge 0\}, \ i\in \{0,...,m-1\}. 
\end{equation}

\begin{definition}[HOCBF~\cite{xiao2021high}]
\label{def:HOCBF}
Let $\psi_{i}(\boldsymbol{x}),\ i\in \{1,...,m\}$ be defined by \eqref{eq:sequence-f1} and $\mathcal C_{i},\ i\in \{0,...,m-1\}$ be defined by \eqref{eq:sequence-set1}. A function $b:\mathbb{R}^{n}\to\mathbb{R}$ is a High-Order Control Barrier Function (HOCBF) with relative degree $m$ for system \eqref{eq:affine-control-system} if there exist $(m-i)^{th}$ order differentiable class $\kappa$ functions $\alpha_{i},\ i\in \{1,...,m\}$ such that
\begin{equation}
\label{eq:highest-HOCBF}
\begin{split}
\sup_{\boldsymbol{u}\in \mathcal U}[L_{f}^{m}b(\boldsymbol{x})+L_{g}L_{f}^{m-1}b(\boldsymbol{x})\boldsymbol{u}+O(b(\boldsymbol{x}))
+\\
\alpha_{m}(\psi_{m-1}(\boldsymbol{x}))]\ge 0,
\end{split}
\end{equation}
$\forall \boldsymbol{x}\in \mathcal C_{0}\cap,...,\cap \mathcal C_{m-1},$ where $L_{f}^{m}$ denotes the $m^{th}$ Lie derivative along $f$ and $L_{g}$ denotes the matrix of Lie derivatives along the columns of $g$; 
$O(\cdot)=\sum_{i=1}^{m-1}L_{f}^{i}(\alpha_{m-1}\circ\psi_{m-i-1})(\boldsymbol{x})$ contains the remaining Lie derivatives along $f$ with degree less than or equal to $m-1$. $\psi_{i}(\boldsymbol{x})\ge0$ is referred to as the $i^{th}$ order HOCBF inequality (constraint in optimization). We assume that $L_{g}L_{f}^{m-1}b(\boldsymbol{x})\boldsymbol{u}\ne0$ on the boundary of set $\mathcal C_{0}\cap,...,\cap \mathcal C_{m-1}.$ 
\end{definition}

\begin{theorem}[Safety Guarantee~\cite{xiao2021high}]
\label{thm:safety-guarantee}
Given a HOCBF $b(\boldsymbol{x})$ from Def. \ref{def:HOCBF} with corresponding sets $\mathcal{C}_{0}, \dots,\mathcal {C}_{m-1}$ defined by \eqref{eq:sequence-set1}, if $\boldsymbol{x}(0) \in \mathcal {C}_{0}\cap \dots \cap \mathcal {C}_{m-1},$ then any Lipschitz controller $\boldsymbol{u}$ that satisfies the inequality in \eqref{eq:highest-HOCBF}, $\forall t\ge 0$ renders $\mathcal {C}_{0}\cap \dots \cap \mathcal {C}_{m-1}$ forward invariant for system \eqref{eq:affine-control-system}, $i.e., \boldsymbol{x} \in \mathcal {C}_{0}\cap \dots \cap \mathcal {C}_{m-1}, \forall t\ge 0.$
\end{theorem}

\begin{definition}[CLF~\cite{ames2012control}]
\label{def:control-l-f}
A continuously differentiable function $V:\mathbb{R}^{n}\to\mathbb{R}$ is an exponentially stabilizing Control Lyapunov Function (CLF) for system \eqref{eq:affine-control-system} if there exist constants $c_{1}>0, c_{2}>0,c_{3}>0$ such that for $\forall \boldsymbol{x} \in \mathbb{R}^{n}, c_{1}\left \|  \boldsymbol{x} \right \| ^{2} \le V(\boldsymbol{x}) \le c_{2}\left \|  \boldsymbol{x} \right \| ^{2},$
\begin{equation}
\label{eq:clf}
\inf_{\boldsymbol{u}\in \mathcal U}[L_{f}V(\boldsymbol{x})+L_{g}V(\boldsymbol{x})\boldsymbol{u}+c_{3}V(\boldsymbol{x})]\le 0.
\end{equation}
\end{definition}

 Some existing works \cite{nguyen2016exponential},\cite{xiao2021high} combine HOCBFs \eqref{eq:highest-HOCBF} for systems with high relative degree with quadratic costs to form safety-critical optimization problems. The HOCBFs are used to ensure the forward invariance of sets related to safety requirements, therefore guaranteeing safety.
CLFs \eqref{eq:clf} can also be incorporated (see \cite{xiao2021high},\cite{xiao2021adaptive}) if exponential convergence of some states is desired. In these works, the control inputs are the decision (optimization) variables. Time is discretized into intervals, and an optimization problem with constraints given by HOCBFs (hard constraints) and CLFs (soft constraints) is solved in each
time interval. The state value is fixed
at the beginning of each interval, which results in linear constraints for the control - the resulting 
optimization problem is a QP. The optimal control obtained by solving each QP is applied at the beginning of the interval and held constant for the
whole interval. During each interval, the state is updated using dynamics \eqref{eq:affine-control-system}. 

This method, which throughout this paper we will referred to as the CBF-CLF-QP method, works conditioned on the fact that solving the QP at every time interval is feasible. 
However, this is not guaranteed, and, in fact, unlikely to happen, if the control bounds in Eqn. (\ref{eq:control-constraint}) are tight. The authors of \cite{xiao2022sufficient} proposed sufficient conditions to address the feasibility issue. In short, they created a feasibility constraint enforced by a first-order CBF constraint (hard constraint) to avoid conflict between the $m^{th}$ order HOCBF constraint (hard constraint for safety) and control constraints \eqref{eq:control-constraint} (hard constraint for control bounds). They  made sure this first order CBF constraint was also compatible with the hard constraints for safety and control bounds. This method increases the overall feasibility of solving QPs since all hard constraints are compatible with each other. The method was successfully applied to a traffic merging control problem for Connected and Automated Vehicles (CAVs) \cite{xu2022feasibility}. 
However, these sufficient conditions are heavily dependent on the considered dynamics and constraints. 
In this paper, we show how we can find and satisfy sufficient conditions for the feasibility of the QPs given general dynamics and constraints.

\section{Problem Formulation and Approach}
\label{sec:Problem Formulation and Approach}


Our goal is to generate a control strategy for system \eqref{eq:affine-control-system} such that it converges to a desired state, some measure of spent energy is minimized, safety is satisfied, and control limitations are observed. 

\textbf{Objective:} We consider the cost  
\begin{equation}
\label{eq:cost-function-1}
\begin{split}
 J(\boldsymbol{u}(t))=\int_{0}^{T} 
 \| \boldsymbol{u}(t) \| ^{2}dt+p\left \| \boldsymbol{x}(T)-\boldsymbol{x}_{e} \right \| ^{2},
\end{split}
\end{equation}
where $\left \| \cdot \right \|$ denotes the 2-norm of a vector, and $T>0$ denotes the ending time; $p>0$ denotes a weight factor and $\boldsymbol{x}_{e} \in \mathbb{R}^{n}$ is a desired state, which is assumed to be an equilibrium for system \eqref{eq:affine-control-system}. $p\left \| \boldsymbol{x}(T)-\boldsymbol{x}_{e} \right \| ^{2}$ denotes state convergence.

\textbf{Safety Requirement:} System \eqref{eq:affine-control-system} should always satisfy one or more safety requirements of the form: 
\begin{equation}
\label{eq:Safety constraint}
b(\boldsymbol{x})\ge 0, \boldsymbol{x} \in \mathbb{R}^{n}, \forall t \in [0, T],
\end{equation}
where $b:\mathbb{R}^{n}\to\mathbb{R}$ is assumed to be a continuously differentiable equation. 

\textbf{Control Limitations:} The controller $\boldsymbol{u}$ should always satisfy \eqref{eq:control-constraint} for all $t \in [0, T].$

A control policy is \textbf{feasible} if \eqref{eq:Safety constraint} and \eqref{eq:control-constraint} are strictly satisfied $\forall t \in [0, T].$ In this paper, we consider the following problem:

\begin{problem}
\label{prob:SACC-prob}
Find a feasible control policy for system \eqref{eq:affine-control-system} such that cost \eqref{eq:cost-function-1} is minimized.
\end{problem}

\textbf{Approach:} We define a HOCBF to enforce \eqref{eq:Safety constraint}.  We also use a relaxed CLF to realize the state convergence in \eqref{eq:cost-function-1}. Since the cost is quadratic in $\boldsymbol{u}$, we can formulate Prob. \ref{prob:SACC-prob} using CBF-CLF-QPs:
\begin{equation}
\label{eq:optimal control-cost}
\begin{split}
\min_{u(t),\delta(t)} \int_{0}^{T}(\left \| \boldsymbol{u}(t) \right \| ^{2}+p\delta^{2}(t))dt.
\end{split}
\end{equation}
subject to
\begin{subequations}
\label{eq:hard constraints}
\begin{align}
L_{f}^{m}b(\boldsymbol{x})+L_{g}L_{f}^{m-1}&b(\boldsymbol{x})\boldsymbol{u}+O(b(\boldsymbol{x}))
+\alpha_{m}(\psi_{m-1}(\boldsymbol{x}))\ge 0,\label{subeq:HOCBF as 1}\\ 
L_{f}V(\boldsymbol{x})+&L_{g}V(\boldsymbol{x})\boldsymbol{u}+c_{3}V(\boldsymbol{x})\le \delta(t),\label{subeq:CLF as 2}\\
&\boldsymbol{u}_{min}\le \boldsymbol{u} \le \boldsymbol{u}_{max},\label{subeq:control bounds as 3}
\end{align}
\end{subequations}
where $V(\boldsymbol{x}(t))=(\boldsymbol{x}(t)-\boldsymbol{x}_{e})^{T}P(\boldsymbol{x}(t)-\boldsymbol{x}_{e}), P$ is positive definite, $c_{3}>0, p>0$ and $\delta(t) \in \mathbb{R}$ is a relaxation variable (decision variable) that we wish to minimize for less violation of the strict CLF constraint.  $b(\boldsymbol{x})$ has relative degree $m$ and $V(\boldsymbol{x})$ has relative degree 1. The above optimization problem is \textbf{feasible at a given state $\boldsymbol{x}$} if all the constraints define a non-empty set for the decision variables $\boldsymbol{u},\delta.$

The CBF-CLF-QP approach to the above optimization problem, already summarized in Sec. \ref{sec:Preliminaries}, starts by discretizing the time interval $[0,T]$ into several equal intervals $[t_{k},t_{k+1}),t_{0}=0, t_{N}=T,k\in \{0,...,N-1\}$. 
At the beginning of each time interval $t_{k} (k\ge1)$, given 
$\boldsymbol{x}(t_{k})$, we solve the following optimization problem (the CBF-CLF-QP):
\begin{equation}
\label{eq:descritized hard constraints}
\begin{split}
(\boldsymbol{u}^{\ast}(t_{k}),\delta^{\ast}(t_{k}))=\arg \min_{\boldsymbol{u}(t_{k}),\delta(t_{k})} (\left \| \boldsymbol{u}(t_{k}) \right \| ^{2}+p\delta^{2}(t_{k})),
\end{split}
\end{equation}
subject to constraints \eqref{eq:hard constraints} (we initialize $\boldsymbol{x}(t_{0})$ to make it satisfy $\psi_{i}(\boldsymbol{x}(t_{0}))\ge 0, i\in \{1,...,m-1\}$ based on Thm. \ref{thm:safety-guarantee}). Then we apply the optimal controller $\boldsymbol{u}^{\ast}(t_{k})$ to system \eqref{eq:affine-control-system}. We use the value of the state $\boldsymbol{x}(t_{k+1})$ to formulate the next CBF-CLF-QP at $t_{k+1}.$ Repeatedly doing the above process we hope to finally get the discretized optimal control set $\boldsymbol{u}^{\ast}_{[0,T]}$ and state set $\boldsymbol{x}_{[0,T]}.$ However, the CBF-CLF-QPs could easily be infeasible at some $t_{k}$. In other words, after applying the constant vector $\boldsymbol{u}^{\ast}(t_{k-1})$ to system \eqref{eq:affine-control-system} for the time interval $[t_{k-1},t_{k})$, we may end up at a state $\boldsymbol{x}(t_{k})$ where
the HOCBF constraint \eqref{subeq:HOCBF as 1} conflicts with the control bounds \eqref{subeq:control bounds as 3}, which would render the CBF-CLF-QP corresponding to getting $\boldsymbol{u}^{\ast}(t_{k})$ infeasible. One way to deal with this would be to find appropriate hyperparameters (e.g., $p_{t_{k}},P_{t_{k}},c_{3,t_{k}},\alpha_{i,t_{k}}(\cdot)$)
such that the safety requirements and the control limitations are satisfied, i.e., $b(\boldsymbol{x}_{[0,T]})\ge 0,$ and $\boldsymbol{u}_{min}\le \boldsymbol{u}^{\ast}_{[0,T]} \le \boldsymbol{u}_{max}$. However, this is a difficult problem. 
Motivated by \cite{xiao2022sufficient}, our approach is to define a feasibility constraint and use CBF constraints to enforce the feasibility constraint. These CBF constraints will provide sufficient conditions for the feasibility of Prob. \ref{prob:SACC-prob}.

In this paper, we assume that Prob. \ref{prob:SACC-prob} {\em is solvable}. While we do not provide a formal definition for this assumption, we illustrate it through a simple example. Consider a vehicle with 1-D dynamics $\dot{x}=x+u,$ where $\dot{x},x,u$ denote its linear speed, position, and acceleration, respectively. If the control bound is $-1\le u \le 1$, initial position is $x(0)=1.8,$ and there is a tree located at $x_{\text{tree}}=2$, it is impossible for the vehicle to maintain a safe distance from the tree, i.e., keep $x_{\text{tree}}-x\ge0, \forall t \ge 0$, i.e., this problem is unsolvable.  

\section{Feasibility Constraint}
\label{sec:Feasibility Constraint}

We begin with a simple motivation example to illustrate the necessity for a feasibility constraint for the CBF-CLF-QPs. Consider a Simplified Adaptive Cruise Control (SACC) problem for two vehicles with the dynamics of ego vehicle expressed as 
\begin{small}
\begin{equation}
\label{eq:SACC-dynamics}
\underbrace{\begin{bmatrix}
\dot{z}(t) \\
\dot{v}(t) 
\end{bmatrix}}_{\dot{\boldsymbol{x}}(t)}  
=\underbrace{\begin{bmatrix}
 v_{p}-v(t) \\
 0
\end{bmatrix}}_{f(\boldsymbol{x}(t))} 
+ \underbrace{\begin{bmatrix}
  0 \\
  1 
\end{bmatrix}}_{g(\boldsymbol{x}(t))}u(t),
\end{equation}
\end{small}
where $v_{p}>0, v(t)>0$ denote the velocity of the lead vehicle (constant velocity) and ego vehicle, respectively. $z(t)$ denotes the distance between the lead and ego vehicle and $u(t)$ denotes the acceleration (control) of ego vehicle, subject to the control constraints
\begin{equation}
\label{eq:simple-control-constraint}
u_{min}\le u(t) \le u_{max}, \forall t \ge0,
\end{equation}
where $u_{min}<0$ and $u_{max}>0$ are the minimum and maximum control input, respectively.

 For safety, we require that $z(t)$ always be greater than or equal to the safety distance denoted by $l_{p}>0,$ i.e., $z(t)\ge l_{p}, \forall t \ge 0.$ Based on Def. \ref{def:HOCBF}, we define the safety constraint as $\psi_{0}(\boldsymbol{x})\coloneqq b(\boldsymbol{x})=z(t)-l_{p}\ge 0.$ From \eqref{eq:sequence-f1}-\eqref{eq:highest-HOCBF}, since the relative degree of $b(\boldsymbol{x})$ is 2, we have
\begin{equation}
\label{eq:SACC-HOCBF-sequence}
\begin{split}
&\psi_{1}(\boldsymbol{x})\coloneqq v_{p}-v(t)+k_{1}\psi_{0}(\boldsymbol{x})\ge 0
,\\
&\psi_{2}(\boldsymbol{x},\boldsymbol{u})\coloneqq -u(t)+k_{1}(v_{p}-v(t))+k_{2}\psi_{1}(\boldsymbol{x})\ge 0,
\end{split}
\end{equation}
where $\alpha_{1}(\psi_{0}(\boldsymbol{x}))\coloneqq k_{1}\psi_{0}(\boldsymbol{x}), \alpha_{2}(\psi_{1}(\boldsymbol{x}))\coloneqq k_{2}\psi_{1}(\boldsymbol{x}), k_{1}>0, k_{2}>0.$ The constant coefficients $k_{1},k_{2}$ are always chosen small to equip ego vehicle with a conservative control strategy to keep it safe, i.e., smaller $k_{1},k_{2}$ make ego vehicle brake earlier (see \cite{xiao2021high}). Suppose we wish to minimize the energy cost $\int_{0}^{T} u^{2}(t)dt$. We can then formulate the QPs using Eqns. \eqref{eq:optimal control-cost}, \eqref{eq:hard constraints} as described above to get the optimal controller for the SACC problem. However, the optimization problem can easily become infeasible if 
\begin{equation}
u(t)\le k_{1}(v_{p}-v(t))+k_{2}\psi_{1}(\boldsymbol{x})<u_{min}.  
\end{equation}
To avoid this, similar to safety, we can define a feasibility constraint 
\begin{equation}
\label{eq:SACC-feasibility constraint}
b_{F}(\boldsymbol{x})=-u_{min}+k_{1}(v_{p}-v(t))+k_{2}\psi_{1}(\boldsymbol{x})\ge 0
\end{equation}
and enforce it by making $b_{F}$ a CBF. From \eqref{eq:sequence-f1}-\eqref{eq:highest-HOCBF}, since the relative degree of $b_{F}(\boldsymbol{x})$ is 1, we can enforce the satisfaction of the feasibility constraint by imposing a first order CBF constraint as 
\begin{equation}
\label{eq:SACC-HOCBF-feasibility}
\psi_{F}(\boldsymbol{x},\boldsymbol{u})\coloneqq L_{f}b_{F}(\boldsymbol{x})+L_{g}b_{F}(\boldsymbol{x})u(t)+k_{F}b_{F}(\boldsymbol{x})\ge 0,
\end{equation}
where $\alpha_{F}(b_{F}(\boldsymbol{x}))\coloneqq k_{F}b_{F}(\boldsymbol{x}), k_{F}>0$. We add $\psi_{F}(\boldsymbol{x},\boldsymbol{u})\ge 0$ to the QP as a hard constraint to ensure $b_{F}(\boldsymbol{x})\ge 0$. Therefore, $\psi_{2}(\boldsymbol{x},\boldsymbol{u})\ge0$ does not conflict with \eqref{eq:simple-control-constraint}. In \cite{xiao2022sufficient} and \cite{xu2022feasibility},  $k_{F}$ in \eqref{eq:SACC-HOCBF-feasibility} was adjusted to make constraint \eqref{eq:SACC-HOCBF-feasibility} compatible with $\psi_{2}(\boldsymbol{x},\boldsymbol{u})\ge0$ and \eqref{eq:simple-control-constraint}. As a result, the feasible set of inputs under all hard constraints is not empty and the optimization is feasible. 

The method described above makes the problem feasible. However, it is not clear how to generalize this method to more complicated hard constraints and for the case when there are multiple control inputs in \eqref{eq:SACC-HOCBF-feasibility}. Before we introduce our method, we first provide a general definition of a feasibility constraint. 

\begin{assumption}\label{asm:Feasibility Constraint}
Let $b:\mathbb{R}^{n}\to\mathbb{R}$ be a HOCBF candidate as in Def. \ref{def:HOCBF} with relative degree $m\ge1$. We assume that all components of the vector $L_{g}L_{f}^{m-1}b(\boldsymbol{x})$ in the $m^{th}$ order constraint \eqref{eq:highest-HOCBF} do not change sign.
\end{assumption}
 If some component of vector $L_{g}L_{f}^{m-1}b(\boldsymbol{x})$ changes sign over time, we can consider Assumption \ref{asm:Feasibility Constraint} as an event and determine a set of events within the whole time period $[0, T]$. In each event, Assumption \ref{asm:Feasibility Constraint} is satisfied. This approach will transform Prob. \ref{prob:SACC-prob} into an event-triggered safety-critical control problem (see \cite{xiao2022event}), which can be solved when the switching frequency between two events is bounded.

\begin{definition}[Feasibility Constraint]
\label{def:Feasibility Constraint}
Assume we have a HOCBF candidate $b:\mathbb{R}^{n}\to\mathbb{R}$ with relative degree $m$ that satisfies the conditions in Assumption \ref{asm:Feasibility Constraint}. We define
\begin{equation}
\label{eq:Maximum value}
\begin{split}
L_{g}L_{f}^{m-1}b(\boldsymbol{x})\boldsymbol{u}_{M}=\sup_{\boldsymbol{u}\in \mathcal U}[L_{g}L_{f}^{m-1}b(\boldsymbol{x})\boldsymbol{u}],
\end{split}
\end{equation}
where $\mathcal U$ is defined the same as \eqref{eq:control-constraint} and not any component of $\boldsymbol{u}_{min},\boldsymbol{u}_{max}$ from $\mathcal U$ reaches negative infinity or infinity. A constraint
\begin{equation}
\label{eq:Feasibility Constraint}
\begin{split}
b_{F}(\boldsymbol{x})\coloneqq L_{f}^{m}b(\boldsymbol{x})+L_{g}L_{f}^{m-1}b(\boldsymbol{x})\boldsymbol{u}_{M}+
O_{F}(b(\boldsymbol{x}))+\\
\alpha_{m}(\psi_{m-1}(\boldsymbol{x}))\ge 0,
\end{split}
\end{equation}
where $b_{F}:\mathbb{R}^{n}\to\mathbb{R}$, is called a feasibility constraint and $\boldsymbol{u}_{min}\le \boldsymbol{u}_{M} \le \boldsymbol{u}_{max}$. $O_{F}(\cdot )$ is a function similar to $O(\cdot )$ in \eqref{eq:highest-HOCBF}.
\end{definition}
Note that $b_{F}(\boldsymbol{x})$ in the above definition is a CBF candidate. In fact, constraint \eqref{eq:Feasibility Constraint}  under \eqref{eq:Maximum value} is the same as \eqref{eq:highest-HOCBF}. Satisfying \eqref{eq:Feasibility Constraint} means there always exist solutions for $\boldsymbol{u}$ under constraints \eqref{subeq:HOCBF as 1}, \eqref{subeq:control bounds as 3} and Assumption \ref{asm:Feasibility Constraint}.

In \cite{xiao2022sufficient}, the authors introduced sufficient CBF constraints to ensure the feasibility of the QPs for an adaptive cruise control problem, where the expression of the feasibility constraint is similar to \eqref{eq:SACC-feasibility constraint}. We notice that, for this case, only one control input with a constant coefficient is involved, which makes this problem easy. In other words, the method introduced in \cite{xiao2022sufficient} is case-dependent, and cannot handle feasibility constraints with complicated expressions, e.g., when there are many control inputs in the CBFs, and when the coefficients of these control inputs vary over time. 

All these issues will be addressed by finding a method to ensure the feasibility constraint \eqref{eq:Feasibility Constraint}, since this allows for multiple control inputs with time-varying coefficients. With the expression of the feasibility constraint, we plan to find a first order CBF $\psi_{F}$ to ensure \eqref{eq:Feasibility Constraint} without conflicting with other hard constraints, which will be illustrated in Sec. \ref{sec:Controller Design for Safety and Feasibility}.

\section{Controller Design for Safety and Feasibility} 
\label{sec:Controller Design for Safety and Feasibility}

In this section, we develop a sufficient condition for CBF-CLF-QP feasibility by defining a CBF constraint $\psi_{F}\ge0$. Motivated by \cite{liu2023auxiliary}, given a CBF candidate $b_{F}:\mathbb{R}^{n}\to\mathbb{R}$ with relative degree 1 for system \eqref{eq:affine-control-system}, we can define a positive auxiliary function $\mathcal{A}(a(t)):\mathbb{R}\to\mathbb{R}^{+}$ based on a time varying auxiliary variable $a(t)$, which is combined with $b_{F}(\boldsymbol{x}(t))$ as $\mathcal{A}(a(t))b_{F}(\boldsymbol{x}(t))$ and used to adaptively enhance the compatibility of hard constraints under CBF-CLF-QPs. The modified function $\mathcal{A}(a(t))b_{F}(\boldsymbol{x}(t))$ has relative degree 1 with respect to system \eqref{eq:affine-control-system}. Based on Thm. \ref{thm:safety-guarantee}, we propose the following lemma:

\begin{lemma}
\label{lemma:Auxiliary equation CBF}
Given a CBF candidate $b_{F}(\boldsymbol{x})$ from Def. \ref{def:Feasibility Constraint} and an auxiliary function $\mathcal{A}(a)$ stated above with the corresponding set $\mathcal C_{F}\coloneqq \{\boldsymbol{x}\in\mathbb{R}^{n}:b_{F}(\boldsymbol{x})> 0\}$, if $\boldsymbol{x}(0) \in \mathcal {C}_{F},$ then any Lipschitz controller $\boldsymbol{u}$ that satisfies the constraint:
\begin{equation}
\label{eq:modified CBF constraint}
\begin{split}
\psi_{F}(\boldsymbol{x},\boldsymbol{u},\boldsymbol{a},\dot{\boldsymbol{a}})=\sup_{\boldsymbol{u}\in \mathcal U}[\frac{\partial \mathcal{A}(a)}{\partial a}\dot{a}b_{F}(\boldsymbol{x})+ \mathcal{A}(a)(L_{f}b_{F}(\boldsymbol{x})+\\L_{g}b_{F}(\boldsymbol{x})\boldsymbol{u})+\alpha_{F}(\mathcal{A}(a)b_{F}(\boldsymbol{x}))]\ge \epsilon,
\end{split}
\end{equation}
$\forall t\ge 0$ renders $\mathcal {C}_{F}$ forward invariant for system \eqref{eq:affine-control-system}, $i.e., \boldsymbol{x} \in \mathcal {C}_{F}, \forall t\ge 0$, where $\epsilon$ is a positive constant which can be arbitrarily small. 
\end{lemma}

\begin{proof}
If $b_{F}(\boldsymbol{x})$ is a CBF candidate that is first order differentiable and at $t=t_{s}\ge 0, \alpha_{F}(\mathcal{A}(a)b_{F}(\boldsymbol{x}))_{t=t_{s}}=\epsilon>0$, then based on Def. \ref{def:class-k-f}, we have $\mathcal{A}(a)b_{F}(\boldsymbol{x})_{t=t_{s}}=\varepsilon>0.$ From \eqref{eq:modified CBF constraint}, we have
\begin{equation}
\begin{split}
\frac{\partial \mathcal{A}(a)b_{F}(\boldsymbol{x})}{\partial t}_{t=t_{s}}=\sup_{\boldsymbol{u}\in \mathcal U}[\frac{\partial \mathcal{A}(a)}{\partial a}\dot{a}b_{F}(\boldsymbol{x})+ \mathcal{A}(a)(L_{f}b_{F}(\boldsymbol{x})+\\L_{g}b_{F}(\boldsymbol{x})\boldsymbol{u})]_{t=t_{s}}\ge 0,
\end{split}
\end{equation}
which is equivalent to make $\mathcal{A}(a)b_{F}(\boldsymbol{x})\ge \varepsilon> 0, \alpha_{F}(\mathcal{A}(a)b_{F}(\boldsymbol{x}))\ge \epsilon>0, \forall t\ge 0.$ Since $\mathcal{A}(a)>0$, we have $b_{F}(\boldsymbol{x})> 0,$ which shows the forward invariance of $\mathcal {C}_{F}.$  
\end{proof}
With Lemma. \ref{lemma:Auxiliary equation CBF}, we can formulate the new constraint \eqref{eq:modified CBF constraint} as one sufficient hard constraint in \eqref{eq:hard constraints}.

\begin{remark}
\label{rem:positive auxiliary equation} 
There are many methods to ensure $\mathcal{A}(a)$ is always positive. One intuitive way is to use an exponential function, e.g., $\mathcal{A}(a)=e^{a}.$ We can also mimic the method from \cite{liu2023auxiliary} to define auxiliary inputs and use them as decision variables in cost \eqref{eq:optimal control-cost} to ensure $\mathcal{A}(a)=a>0$. However, this method will involve more extra constraints, which will affect the feasibility of the optimization. Therefore, we only consider $\mathcal{A}  (a)=e^{a}$ in Sec. \ref{sec:Case Study and Simulations}.
\end{remark}


\begin{theorem}
\label{thm:feasibility and safety}
Let $b:\mathbb{R}^{n}\to\mathbb{R}$ be a HOCBF candidate with relative degree $m$ that satisfies the conditions stated in Assumption \ref{asm:Feasibility Constraint}. 
Let $b_{F}(\boldsymbol{x})$ be a CBF candidate as in Def. \ref{def:Feasibility Constraint}, and $\mathcal{A}(a)$
be an auxiliary function as in Lemma \ref{lemma:Auxiliary equation CBF}, with corresponding sets $\mathcal{C}_{0}, \dots,\mathcal {C}_{m-1}$ defined by \eqref{eq:sequence-set1} and another set defined by $\mathcal C_{A}\coloneqq \{\boldsymbol{x}\in\mathbb{R}^{n}:\mathcal{A}(a)b_{F}(\boldsymbol{x})> 0\}$. Assume $\boldsymbol{x}(0) \in \mathcal {C}_{0}\cap \dots \cap \mathcal {C}_{m-1}\cap \mathcal C_{A}$. Then any Lipschitz controller $\boldsymbol{u}$ that satisfies constraints \eqref{eq:modified CBF constraint}, \eqref{subeq:HOCBF as 1} and \eqref{subeq:control bounds as 3} renders $\mathcal {C}_{0}$ forward invariant for system \eqref{eq:affine-control-system}, $i.e., \boldsymbol{x} \in \mathcal {C}_{0}, \forall t\ge 0.$ If  we define
\begin{equation}
\label{eq:pre-defined variables}
\begin{split}
\dot{a}=-\frac{\partial a}{\partial \mathcal{A}(a)}\frac{(\mathcal{A}(a)L_{f}b_{F}(\boldsymbol{x})+\lambda)}{b_{F}(\boldsymbol{x})},\\
\lambda=\mathcal{A}(a)L_{g}b_{F}(\boldsymbol{x})\boldsymbol{u}_{M},
\end{split}
\end{equation}
 and $\frac{\partial \mathcal{A}(a)}{\partial a}\ne 0,\forall t\ge 0,$ then constraints \eqref{eq:modified CBF constraint}, \eqref{subeq:HOCBF as 1} and \eqref{subeq:control bounds as 3} are always compatible with each other, i.e., the optimization problem using CBF-CLF-QPs is always feasible, and $b$ becomes a valid HOCBF ($a$, and $\lambda$ are two predefined time varying variables). 
\end{theorem}
\begin{proof}
Given $\boldsymbol{x}(0) \in \mathcal {C}_{0}\cap \dots \cap \mathcal {C}_{m-1}\cap \mathcal C_{A},$ since $\boldsymbol{x} \in \mathcal {C}_{m-1} \cap \mathcal {C}_{A}, \forall t \ge0$ is ensured by satisfying constraints \eqref{eq:modified CBF constraint}, \eqref{subeq:HOCBF as 1} and \eqref{subeq:control bounds as 3}, based on Lemma. \ref{lemma:Auxiliary equation CBF}, $b_{F}(\boldsymbol{x})> 0$ is guaranteed $\forall t \ge0.$ Based on Thm. \ref{thm:safety-guarantee}, $\mathcal {C}_{0}\cap \dots \cap \mathcal {C}_{m-1}$ is guaranteed forward invariant for system \eqref{eq:affine-control-system}, thus $b(\boldsymbol{x})\ge 0$ is guaranteed $\forall t \ge0.$ Since $b_{F}(\boldsymbol{x})> 0$ is guaranteed, based on \eqref{eq:Maximum value}, we have 
\begin{equation}
\label{eq:Feasibility Constraint 2}
\begin{split}
b_{F}(\boldsymbol{x})=\sup_{\boldsymbol{u}\in \mathcal U}[L_{g}L_{f}^{m-1}b(\boldsymbol{x})\boldsymbol{u}]+L_{f}^{m}b(\boldsymbol{x})+ \\O_{F}(b(\boldsymbol{x}))+
\alpha_{m}(\psi_{m-1}(\boldsymbol{x}))> 0,
\end{split}
\end{equation}
which shows that constraint \eqref{subeq:HOCBF as 1} is compatible with \eqref{subeq:control bounds as 3}. With predefined $\dot{a}$ and $\lambda$ introduced above, we can rewrite \eqref{eq:modified CBF constraint} into the constraint
\begin{small}
\begin{equation}
\label{eq:modified CBF constraint 2}
\begin{split}
\sup_{\boldsymbol{u}\in \mathcal U}[L_{g}b_{F}(\boldsymbol{x})\boldsymbol{u}-L_{g}b_{F}(\boldsymbol{x})\boldsymbol{u}_{M}+
\frac{\alpha_{F}(\mathcal{A}(a)b_{F}(\boldsymbol{x}))}{\mathcal{A}(a)}]\ge \frac{\epsilon}{\mathcal{A}(a)} ,
\end{split}
\end{equation}
\end{small}
which is always satisfied since the constraint $\alpha_{F}(\mathcal{A}(a)b_{F}(\boldsymbol{x}))\ge \epsilon$ is guaranteed satisfied (based on the proof of Lemma. \ref{lemma:Auxiliary equation CBF}) and  
\begin{equation}
\label{eq:modified CBF constraint 3}
\sup_{\boldsymbol{u}\in \mathcal U}[L_{g}b_{F}(\boldsymbol{x})\boldsymbol{u}-L_{g}b_{F}(\boldsymbol{x})\boldsymbol{u}_{M}]\ge 0
\end{equation}
is guaranteed to be satisfied. Obviously controller $\boldsymbol{u}=\boldsymbol{u}_{M}$ satisfies constraints  \eqref{eq:modified CBF constraint 2} and \eqref{eq:Feasibility Constraint 2}. If \eqref{eq:modified CBF constraint 2} is satisfied, \eqref{eq:modified CBF constraint} is definitely satisfied, therefore constraint \eqref{eq:modified CBF constraint} is compatible with \eqref{subeq:HOCBF as 1}  and \eqref{subeq:control bounds as 3}. Consequently, the optimization problem using CBF-CLF-QPs with constraints \eqref{subeq:HOCBF as 1}, \eqref{subeq:control bounds as 3} and \eqref{eq:modified CBF constraint} is always feasible.  
\end{proof}

 Previous works \cite{zeng2021safety}, \cite{liu2023auxiliary}, \cite{xiao2021adaptive} introduced relaxation variables or auxiliary control inputs in the CBF constraint \eqref{subeq:HOCBF as 1} as decision variables, which make \eqref{subeq:HOCBF as 1} a soft constraint to enhance the feasibility. As opposed to the above methods, we introduce a positive auxiliary function $\mathcal{A}(a)$ in Thm. \ref{thm:feasibility and safety} for the feasibility constraint, which contains an auxiliary variable $a$. The derivative of the auxiliary variable $\dot{a}$ is defined as a function with arguments $a,\boldsymbol{x},\lambda$ to make the CBF constraint \eqref{eq:modified CBF constraint} compatible with \eqref{subeq:HOCBF as 1}, \eqref{subeq:control bounds as 3}, while guaranteeing safety. All these constraints are hard constraints. 

We define $\dot{a}$ and $\lambda$  precisely to align the feasible regions of constraints \eqref{eq:modified CBF constraint} and \eqref{subeq:HOCBF as 1}. Simultaneously, under the action of constraint \eqref{eq:modified CBF constraint}, the feasible regions of constraints \eqref{subeq:HOCBF as 1} and \eqref{subeq:control bounds as 3} are guaranteed to overlap, ensuring simultaneous overlap of the feasible regions for constraints \eqref{eq:modified CBF constraint}, \eqref{subeq:HOCBF as 1}, and \eqref{subeq:control bounds as 3}. This guarantees a feasible solution ($\boldsymbol{u}=\boldsymbol{u}_{M}$) for the optimization problem under these constraints.

\section{Case Study and Simulations}
\label{sec:Case Study and Simulations}

In this section, we consider the Adaptive Cruise Control (ACC) problem for a heterogeneous platoon (3 vehicles), which is more realistic than the SACC problem introduced in Sec. \ref{sec:Feasibility Constraint} and the case study from \cite{ames2016control}, \cite{xiao2021adaptive}.
\subsection{Vehicle Dynamics}
We consider a nonlinear vehicle dynamics in the form
\begin{small}
\begin{equation}
\label{eq:ACC-dynamics}
\underbrace{\begin{bmatrix}
\dot{x}_{j}(t) \\
\dot{v}_{j}(t) 
\end{bmatrix}}_{\dot{\boldsymbol{x}}_{j}(t)}  
=\underbrace{\begin{bmatrix}
 v_{j}(t) \\
 -\frac{1}{M_{j}}F_{r}(v_{j}(t))
\end{bmatrix}}_{f(\boldsymbol{x}_{j}(t))} 
+ \underbrace{\begin{bmatrix}
  0 \\
  \frac{1}{M_{j}} 
\end{bmatrix}}_{g(\boldsymbol{x}_{j}(t))}u_{j}(t),
\end{equation}
\end{small}
where $M_{j}$ denotes the mass of the $j^{th}$ vehicle, $j=1,2,3$ and $F_{r}(v_{j}(t))=f_{0}sgn(v_{j}(t))+f_{1}v_{j}(t)+f_{2}v^{2}_{j}(t)$ is the resistance force as in \cite{Khalil:1173048};  $f_{0},f_{1},f_{2}$ are positive scalars determined empirically and $v_{j}(t)>0$ denotes the velocity of the vehicle; $x_{j}(t),u_{j}(t)$ denote the position and acceleration of the vehicle, respectively. The first term in $F_{r}(t)$ denotes the Coulomb friction force, the second term denotes the viscous friction force and the last term denotes the aerodynamic drag.

\subsection{Vehicle Limitations}
Vehicle limitations include vehicle constraints on safe distance, speed and acceleration. We consider 3 vehicles driving along the same direction in a line. The first vehicle leads the second vehicle and the second vehicle leads the third vehicle.

\textbf{Safe Distance Constraint:} The distance is considered safe if $x_{j-1}(t)-x_{j}(t)>l_{p},j=2,3$ is satisfied $\forall t \in [0,T]$, where $l_{p}$ denotes the minimum distance two vehicles should maintain, and $j$ is the index of the second and third vehicles.  

\textbf{Speed Constraint:} The second and third vehicles should achieve desired speeds $v_{j,d}>0$, $j=2,3$, respectively.

\textbf{Acceleration Constraint:} The second and third vehicles should minimize the following cost
\begin{small}
\begin{equation}
\label{eq:minimal-u}
\min_{u_{j}(t)} \int_{0}^{T}(\frac{u_{j}(t)-F_{r}(v_{j}(t))}{M_{j}})^{2}dt 
\end{equation}
\end{small}
when the acceleration is constrained in the form 
\begin{equation}
\label{eq:constraint-u}
-c_{j,d}M_{j}g\le u_{j}(t) \le c_{j,a}M_{j}g, \forall t \in [0,T], 
\end{equation}
where $g$ denotes the gravity constant, $c_{j,d}>0$ and $c_{j,a}>0$ are deceleration and acceleration coefficients respectively, $j=2,3.$

\begin{problem}
\label{prob:ACC-prob}
Determine the optimal controllers for the second and third vehicles governed by dynamics \eqref{eq:ACC-dynamics}, subject to the vehicle constraints on safe distance, speed and acceleration. 
\end{problem}

We consider a \textbf{decentralized} optimal control framework for the platoon, i.e., the kinematic information of the lead vehicle is generated by solving Prob. \ref{prob:ACC-prob} and assumed to be known by the following vehicle. Since there is no vehicle leading the first vehicle, we define the controller for the first vehicle as
\begin{equation}
\label{eq:first vehicle controller}
u_{1}(t)=2M_{1}\sin(2\pi t)+F_{r}(v_{1}(t)),
\end{equation}
which represents the swift change of control strategy of the first vehicle. To satisfy the constraint on speed, we define a CLF $V_{j}(\boldsymbol{x}(t)) \coloneqq(v_{j}(t)-v_{j,d})^{2}$ with $c_{j,1}=c_{j,2}=1$ to stabilize $v_{j}(t)$ to $v_{j,d}$ and formulate the relaxed constraint in \eqref{eq:clf} as
\begin{equation}
\label{eq:ACC-clf}
L_{f}V_{j}(\boldsymbol{x}_{j}(t))+L_{g}V_{j}(\boldsymbol{x}_{j}(t))u_{j}(t)+c_{j,3}V_{j}(\boldsymbol{x}_{j}(t))\le \delta_{j}(t), 
\end{equation}
where $\delta_{j}(t)$ is a relaxation that makes \eqref{eq:ACC-clf} a soft constraint.

To satisfy the constraints on safety distance and acceleration, we define a continuous function $b_{j}(\boldsymbol{x}_{j}(t))=x_{j-1}(t)-x_{j}(t)-l_{p}$ as a HOCBF to guarantee $b_{j}(\boldsymbol{x}_{j}(t))\ge 0$ and constraint \eqref{eq:constraint-u}. To ensure there exists at least one solution to the  optimization Prob. \ref{prob:ACC-prob}, we define a continuous function $b_{j,F}(\boldsymbol{x}_{j}(t))\ge 0$ to guarantee feasibility, and then formulate all constraints mentioned above into QPs to get the optimal controller. The parameters are $v_{1}(0)=13.89m/s,v_{2}(0)=8 m/s, v_{3}(0)=14m/s, v_{2,d}=24m/s, v_{3,d}=25m/s, M_{1}=1500kg, M_{2}=1650kg, M_{3}=1550kg, g=9.81m/s^{2}, x_{1}(0)=0m, x_{2}(0)=-100m, x_{3}(0)=-190m, l_{p}=10m, f_{0}=0.1N, f_{1}=5Ns/m, f_{2}=0.25Ns^{2}/m, c_{2,a}=0.4, c_{3,a}=0.35.$
\subsection{Implementation with Auxiliary-Function-Based CBFs}

Let $b_{j}(\boldsymbol{x}_{j}(t))=x_{j-1}(t)-x_{j}(t)-l_{p}$. The relative degree of $b_{j}(\boldsymbol{x}_{j}(t))$ with respect to dynamics \eqref{eq:ACC-dynamics} is 2. The HOCBFs are then defined as 
\begin{equation}
\label{eq:HOCBF-sequence-ACC}
\begin{split}
&\psi_{j,0}(\boldsymbol{x}_{j})\coloneqq b_{j}(\boldsymbol{x}_{j}),\\
&\psi_{j,1}(\boldsymbol{x}_{j})\coloneqq L_{f}b_{j}(\boldsymbol{x}_{j})+k_{j,1}b_{j}(\boldsymbol{x}_{j}),\\
&\psi_{j,2}(\boldsymbol{x}_{j},\boldsymbol{u}_{j})\coloneqq L_{f}^{2}b_{j}(\boldsymbol{x}_{j})+L_{f}L_{g}b_{j}(\boldsymbol{x}_{j})u_{j}+\\
&k_{j,1}L_{f}b_{j}(\boldsymbol{x}_{j})+k_{j,2}\psi_{j,1}(\boldsymbol{x}_{j}),
\end{split}
\end{equation}
where $\alpha_{j,1}(\cdot),\alpha_{j,2}(\cdot)$ are set as linear functions. We define $\varphi_{j,0}(\boldsymbol{x}_{j},\boldsymbol{a}_{j})\coloneqq \mathcal{A}_{j}(a_{j})b_{j,F}(\boldsymbol{x}_{j})> 0$ as the modified feasibility constraint from \eqref{eq:Feasibility Constraint}, $\psi_{j,F}(\boldsymbol{x}_{j},\boldsymbol{u}_{j},\boldsymbol{a}_{j},\dot{\boldsymbol{a}}_{j})\ge \epsilon_{j}$ from \eqref{eq:modified CBF constraint} and $\mathcal{A}_{j}(a_{j})=e^{a_{j}}.$ The auxiliary-function-based CBFs are defined as
\begin{small}
\begin{equation}
\label{eq:CBF-sequence-ACC}
\begin{split}
&\varphi_{j,0}(\boldsymbol{x}_{j},\boldsymbol{a}_{j})\coloneqq e^{a_{j}}(L_{f}^{2}b_{j}(\boldsymbol{x}_{j})+\\
&[L_{f}L_{g}b_{j}(\boldsymbol{x}_{j})u_{j}]_{max}+k_{j,1}L_{f}b_{j}(\boldsymbol{x}_{j})+k_{j,2}\psi_{j,1}(\boldsymbol{x}_{j})),\\
&\psi_{j,F}(\boldsymbol{x}_{j},\boldsymbol{u}_{j},\boldsymbol{a}_{j},\dot{\boldsymbol{a}}_{j})\coloneqq \dot{\varphi}_{j,0}(\boldsymbol{x}_{j},\boldsymbol{u}_{j},\boldsymbol{a}_{j},\dot{\boldsymbol{a}}_{j})+l_{j,F}\varphi_{j,0}(\boldsymbol{x}_{j},\boldsymbol{a}_{j}),\\
\end{split}
\end{equation}
\end{small}
where $\dot{a}_{j}=-\frac{(L_{f}b_{j,F}(\boldsymbol{x}_{j})+\lambda_{j})}{b_{j,F}(\boldsymbol{x}_{j})}, u_{jM}=-c_{j,d}M_{j}g, \lambda_{j}=e^{a_{j}}L_{g}b_{j,F}(\boldsymbol{x}_{j})u_{jM}.$ $\alpha_{j,F}(\cdot)$ is defined as a linear function. The derivative of the resistance force for two vehicles makes the equation of $\psi_{j,F}$ complicated and calls for the introduction of auxiliary adaptive $\dot{a}_{j}$. By formulating the constraints from HOCBFs \eqref{eq:HOCBF-sequence-ACC}, auxiliary-function-based CBFs \eqref{eq:CBF-sequence-ACC}, CLF \eqref{eq:ACC-clf} and acceleration \eqref{eq:constraint-u}, we can define the cost function for the QP as
 \begin{small}
\begin{equation}
\label{eq:Auxiliary-equation based CBF cost}
\begin{split}
\min_{u_{j}(t),\delta_{j}(t)} \int_{0}^{T}[(\frac{u_{j}(t)-F_{r}(v_{j}(t))}{M_{j}})^{2}+p_{j}\delta_{j}(t)^{2}]dt.
\end{split}
\end{equation}
\end{small}
The remaining parameters are set as $a_{1}(0)=a_{2}(0)=1, c_{2,3}=c_{3,3}=1, p_{2}=p_{3}=1000,\epsilon_{2}=\epsilon_{3}=10^{-10}.$

\subsection{Implementation with HOCBFs without Feasibility Constraint}

As a benchmark, we consider the ``traditional" optimization problem without the feasibility constraint, and formulated without the auxiliary-function-based CBFs \eqref{eq:CBF-sequence-ACC}. In other words, the cost function is   \eqref{eq:Auxiliary-equation based CBF cost} and the constraints come from HOCBFs \eqref{eq:HOCBF-sequence-ACC}, CLF \eqref{eq:ACC-clf}, and acceleration \eqref{eq:constraint-u}.  All the corresponding parameters are set to the same values as above. 

\subsection{Simulation Results}

In this subsection, we show how our proposed auxiliary-function-based CBF method guarantees feasibility and safety and outperforms the benchmark described above, which does not use the feasibility constraint. 

We consider Prob. \ref{prob:ACC-prob} with different control bounds \eqref{eq:constraint-u}, and implement HOCBFs as safety constraints with or without feasibility constraint for solving Prob. \ref{prob:ACC-prob} in MATLAB. We use ode45 to integrate the dynamics for every $0.1s$ time-interval and quadprog to solve the QPs. The proposed method shows varying degrees of adaptivity to different lower control bounds in terms of feasibility, safety and optimality.

\begin{figure*}[t]
    \vspace{3mm}
    \centering
    \begin{subfigure}[t]{0.32\linewidth}
        \centering
        \includegraphics[width=1\linewidth]{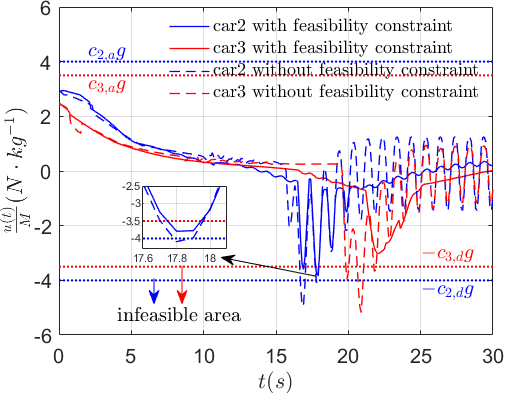}
        \caption{Acceleration of the $2^{nd},3^{rd}$ vehicles vs. time.}
        \label{fig:control input 1}
    \end{subfigure}
    \begin{subfigure}[t]{0.32\linewidth}
        \centering
        \includegraphics[width=1\linewidth]{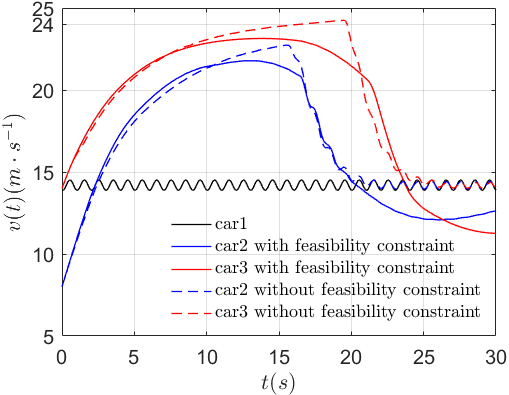}
        \caption{Velocity of the $1^{st},2^{nd},3^{rd}$ vehicles vs. time.}
        \label{fig:velocity 1}
    \end{subfigure}  
    \begin{subfigure}[t]{0.32\linewidth}
        \centering
        \includegraphics[width=1\linewidth]{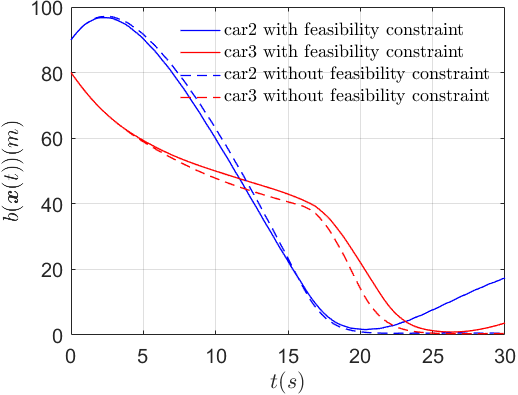}
       \caption{$b(\boldsymbol{x})$ of the $2^{nd},3^{rd}$ vehicles vs. time.}
        \label{fig:distance 1}
    \end{subfigure}
    \caption{Case 1-Feasibility constraint enhances feasibility for solving Prob. \ref{prob:ACC-prob}. The hyperparameters are set as $k_{2,1}=k_{2,2}=k_{3,1}=k_{3,2}=1,l_{2,F}=l_{3,F}=0.1.$ The lower control bounds are $c_{2,d}=0.4,c_{3,d}=0.35.$ Note that without feasibility constraint, the deceleration of the $2^{nd}$(blue), $3^{rd}$(red) vehicles exceeds deceleration bound, therefore the control bounds are violated (as shown by dashed lines in (a)), which causes infeasibility even though the safety is satisfied as shown in (c).} 
    \label{fig:feasibility enhanced}
\end{figure*}

\begin{figure*}[t]
    \centering
    \begin{subfigure}[t]{0.32\linewidth}
        \centering
        \includegraphics[width=1\linewidth]{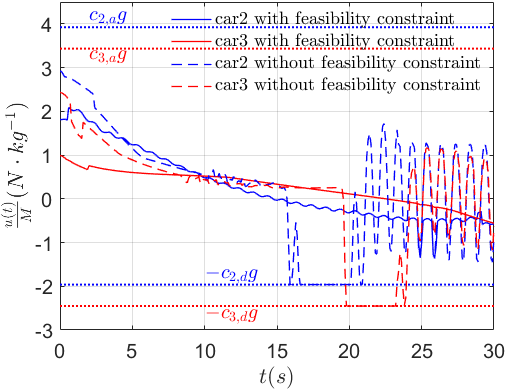}
        \caption{Acceleration of the $2^{nd},3^{rd}$ vehicles vs. time.}
        \label{fig:control input 2}
    \end{subfigure}
    \begin{subfigure}[t]{0.32\linewidth}
        \centering
        \includegraphics[width=1\linewidth]{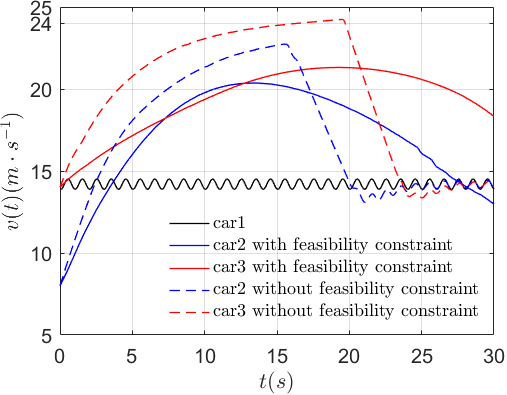}
        \caption{Velocity of the $1^{st},2^{nd},3^{rd}$ vehicles vs. time.}
        \label{fig:velocity 2}
    \end{subfigure}  
    \begin{subfigure}[t]{0.32\linewidth}
        \centering
        \includegraphics[width=1\linewidth]{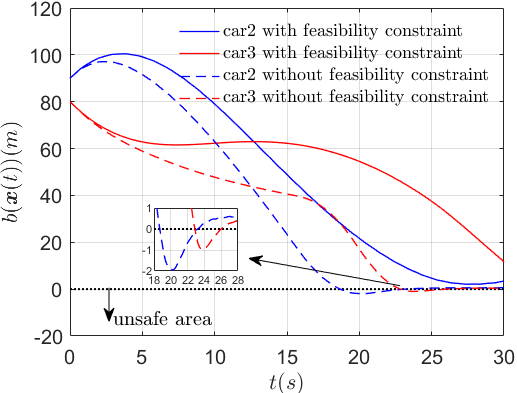}
        \caption{$b(\boldsymbol{x})$ of the $2^{nd},3^{rd}$ vehicles vs. time.}
        \label{fig:distance 2}
    \end{subfigure}
    \caption{Case 2-Feasibility constraint enhances safety for solving Prob. \ref{prob:ACC-prob}. The hyperparameters are set as $k_{2,1}=k_{2,2}=k_{3,1}=k_{3,2}=1,l_{2,F}=l_{3,F}=0.05.$ The lower control bounds are $c_{2,d}=0.2,c_{3,d}=0.25$ (tighter than Fig. \ref{fig:feasibility enhanced}). Note that without feasibility constraint, the $b(\boldsymbol{x})$ of the $2^{nd}$(blue), $3^{rd}$(red) vehicles exceeds safety bound, therefore the safe distance between vehicles can be negative (as shown by dashed lines in (c)), which causes danger even though the control bounds are satisfied as shown in (a).} 
    \label{fig:safety enhanced}
\end{figure*}
We compare two CBFs-based methods in terms of the feasibility of the corresponding QPs. The only difference between the two methods is that one method additionally uses our proposed auxiliary-function-based CBFs to enforce feasibility constraint, while another one is without feasibility constraint. In Fig. \ref{fig:feasibility enhanced}, we keep hyperparameters of safety related HOCBFs the same for two methods as $k_{2,1}=k_{2,2}=k_{3,1}=k_{3,2}=1.$ Two extra hyperparameters are set as $l_{2,F}=l_{3,F}=0.1$ for the feasibility-constraint-related CBFs. The lower control bounds for two vehicles are: $-c_{2,d}M_{2}g, -c_{3,d}M_{3}g$ where $c_{2,d}=0.4,c_{3,d}=0.35.$ In Fig. \ref{fig:control input 1}, it shows that solving QPs for the second and third vehicles is always feasible (denoted by solid lines) with the feasibility constraint since the acceleration is always within bounds
while the QPs will become infeasible (denoted by dashed lines, starting from $t=16.7s, 17.8s, 19.9s, 20.7s, 21.8s.$) without this constraint since the deceleration of two vehicles exceeds deceleration bound. The reason of the effectiveness of the feasibility constraint can be found in Fig. \ref{fig:velocity 1}. With feasibility constraint, the vehicle tends to brake earlier, therefore reaches a smaller peak velocity to avoid a delayed steep deceleration. Since $b(\boldsymbol{x})\ge 0,$ the safe distance $l_{p}$ is always maintained for two vehicles $\forall t \in [0, 30s],$ which can be seen in Fig. \ref{fig:distance 1} and shows the safety is always guaranteed for two methods. 

To test the adaptivity to tighter deceleration bound for the two methods,  we set the lower
control bounds $-c_{2,d}M_{2}g, -c_{3,d}M_{3}g$ for the two vehicles with $c_{2,d}=0.2,c_{3,d}=0.25$. In this case, we compare two CBFs-based methods in terms of safety, without caring about the feasibility. The only difference between the two methods is that one method additionally uses our proposed auxiliary-function-based CBFs to enforce feasibility constraint, while the one without feasibility constraint makes the vehicles brake at the maximum deceleration when the vehicles' decelerations are about to exceed their bounds. In Fig. \ref{fig:safety enhanced}, we keep hyperparameters of safety related HOCBFs the same for the two methods as $k_{2,1}=k_{2,2}=k_{3,1}=k_{3,2}=1.$ Two extra hyperparameters are set as $l_{2,F}=l_{3,F}=0.05$ for the feasibility-constraint-related CBFs. Even though both methods satisfy the acceleration constraint as shown in Fig. \ref{fig:control input 2}, our method effectively maintains a safe distance between the two vehicles, shown by solid lines. Without the feasibility constraint, the vehicles fail to keep the safe distance $l_{p}$ since $b(\boldsymbol{x}) <0$ (denoted by dashed lines from $t=18.7s,22.9s$), as shown in Fig. \ref{fig:distance 2}. Fig. \ref{fig:velocity 2} explains that the feasibility constraint enables earlier braking, allowing the vehicles to keep a longer distance from the corresponding lead vehicle.

We also notice that due to control strategy \eqref{eq:first vehicle controller} used for the first vehicle, the velocity curves in Fig. \ref{fig:velocity 1} and \ref{fig:velocity 2} vibrate frequently (denoted by solid black curve), which might cause the sharp transition of control in the middle shown by dashed curves in Fig. \ref{fig:control input 1} and \ref{fig:control input 2}. Compared to this, our proposed method can generate a smoother optimal controller denoted by solid curves, which might make contribution to reducing more energy cost (increasing optimality).

\section{Conclusion and Future Work}
\label{sec:Conclusion and Future Work}
We propose auxiliary-function-based CBFs as sufficient constraints for safety and feasibility guarantees of constrained optimal control problems, which work for general affine control systems. We have demonstrated the effectiveness of our proposed method in this paper by applying it to an adaptive cruise control problem for a heterogeneous platoon. There are still some scenarios the current method can not perfectly handle, i.e., many other hard constraints are added to optimization problems due to requirements beyond safety, which may lead to conflicts between various constraints. We will address this limitation in future work by creating a more general CBFs-based method with less conservative conditions for constrained optimal control problems.
\bibliographystyle{IEEEtran}
\balance
\bibliography{references.bib}

\begin{thebibliography}{10}
\providecommand{\url}[1]{#1}
\csname url@samestyle\endcsname
\providecommand{\newblock}{\relax}
\providecommand{\bibinfo}[2]{#2}
\providecommand{\BIBentrySTDinterwordspacing}{\spaceskip=0pt\relax}
\providecommand{\BIBentryALTinterwordstretchfactor}{4}
\providecommand{\BIBentryALTinterwordspacing}{\spaceskip=\fontdimen2\font plus
\BIBentryALTinterwordstretchfactor\fontdimen3\font minus
  \fontdimen4\font\relax}
\providecommand{\BIBforeignlanguage}[2]{{%
\expandafter\ifx\csname l@#1\endcsname\relax
\typeout{** WARNING: IEEEtran.bst: No hyphenation pattern has been}%
\typeout{** loaded for the language `#1'. Using the pattern for}%
\typeout{** the default language instead.}%
\else
\language=\csname l@#1\endcsname
\fi
#2}}
\providecommand{\BIBdecl}{\relax}
\BIBdecl

\bibitem{tee2009barrier}
K.~P. Tee, S.~S. Ge, and E.~H. Tay, ``Barrier lyapunov functions for the
  control of output-constrained nonlinear systems,'' \emph{Automatica},
  vol.~45, no.~4, pp. 918--927, 2009.

\bibitem{boyd2004convex}
S.~Boyd, S.~P. Boyd, and L.~Vandenberghe, \emph{Convex optimization}.\hskip 1em
  plus 0.5em minus 0.4em\relax Cambridge university press, 2004.

\bibitem{aubin2011viability}
J.-P. Aubin, A.~M. Bayen, and P.~Saint-Pierre, \emph{Viability theory: new
  directions}.\hskip 1em plus 0.5em minus 0.4em\relax Springer Science \&
  Business Media, 2011.

\bibitem{prajna2007framework}
S.~Prajna, A.~Jadbabaie, and G.~J. Pappas, ``A framework for worst-case and
  stochastic safety verification using barrier certificates,'' \emph{IEEE
  Transactions on Automatic Control}, vol.~52, no.~8, pp. 1415--1428, 2007.

\bibitem{panagou2013multi}
D.~Panagou, D.~M. Stipanovi{\v{c}}, and P.~G. Voulgaris, ``Multi-objective
  control for multi-agent systems using lyapunov-like barrier functions,'' in
  \emph{52nd IEEE Conference on Decision and Control}, 2013, pp. 1478--1483.

\bibitem{wang2016multi}
L.~Wang, A.~D. Ames, and M.~Egerstedt, ``Multi-objective compositions for
  collision-free connectivity maintenance in teams of mobile robots,'' in
  \emph{2016 IEEE 55th Conference on Decision and Control (CDC)}, 2016, pp.
  2659--2664.

\bibitem{glotfelter2017nonsmooth}
P.~Glotfelter, J.~Cort{\'e}s, and M.~Egerstedt, ``Nonsmooth barrier functions
  with applications to multi-robot systems,'' \emph{IEEE control systems
  letters}, vol.~1, no.~2, pp. 310--315, 2017.

\bibitem{ames2016control}
A.~D. Ames, X.~Xu, J.~W. Grizzle, and P.~Tabuada, ``Control barrier function
  based quadratic programs for safety critical systems,'' \emph{IEEE
  Transactions on Automatic Control}, vol.~62, no.~8, pp. 3861--3876, 2016.

\bibitem{ames2012control}
A.~D. Ames, K.~Galloway, and J.~W. Grizzle, ``Control lyapunov functions and
  hybrid zero dynamics,'' in \emph{2012 IEEE 51st IEEE Conference on Decision
  and Control (CDC)}, 2012, pp. 6837--6842.

\bibitem{nguyen2016exponential}
Q.~Nguyen and K.~Sreenath, ``Exponential control barrier functions for
  enforcing high relative-degree safety-critical constraints,'' in \emph{2016
  American Control Conference (ACC)}, 2016, pp. 322--328.

\bibitem{xiao2021high}
W.~Xiao and C.~Belta, ``High-order control barrier functions,'' \emph{IEEE
  Transactions on Automatic Control}, vol.~67, no.~7, pp. 3655--3662, 2021.

\bibitem{isaly2020zeroing}
A.~Isaly, B.~C. Allen, R.~G. Sanfelice, and W.~E. Dixon, ``Zeroing control
  barrier functions for safe volitional pedaling in a motorized cycle,''
  \emph{IFAC-PapersOnLine}, vol.~53, no.~5, pp. 218--223, 2020.

\bibitem{khazoom2022humanoid}
C.~Khazoom, D.~Gonzalez-Diaz, Y.~Ding, and S.~Kim, ``Humanoid self-collision
  avoidance using whole-body control with control barrier functions,'' in
  \emph{2022 IEEE-RAS 21st International Conference on Humanoid Robots
  (Humanoids)}, 2022, pp. 558--565.

\bibitem{cavorsi2022multi}
M.~Cavorsi, B.~Capelli, L.~Sabattini, and S.~Gil, ``Multi-robot adversarial
  resilience using control barrier functions,'' in \emph{Robotics: Science and
  Systems}, 2022.

\bibitem{zeng2021safety2}
J.~Zeng, B.~Zhang, and K.~Sreenath, ``Safety-critical model predictive control
  with discrete-time control barrier function,'' in \emph{2021 American Control
  Conference (ACC)}, 2021, pp. 3882--3889.

\bibitem{liu2023iterative}
S.~Liu, J.~Zeng, K.~Sreenath, and C.~A. Belta, ``Iterative convex optimization
  for model predictive control with discrete-time high-order control barrier
  functions,'' in \emph{2023 American Control Conference (ACC)}, 2023, pp.
  3368--3375.

\bibitem{dawson2022safe}
C.~Dawson, Z.~Qin, S.~Gao, and C.~Fan, ``Safe nonlinear control using robust
  neural lyapunov-barrier functions,'' in \emph{Conference on Robot
  Learning}.\hskip 1em plus 0.5em minus 0.4em\relax PMLR, 2022, pp. 1724--1735.

\bibitem{du2023reinforcement}
D.~Du, S.~Han, N.~Qi, H.~B. Ammar, J.~Wang, and W.~Pan, ``Reinforcement
  learning for safe robot control using control lyapunov barrier functions,''
  in \emph{2023 IEEE International Conference on Robotics and Automation
  (ICRA)}, 2023, pp. 9442--9448.

\bibitem{gurriet2018online}
T.~Gurriet, M.~Mote, A.~D. Ames, and E.~Feron, ``An online approach to active
  set invariance,'' in \emph{2018 IEEE Conference on Decision and Control
  (CDC)}, 2018, pp. 3592--3599.

\bibitem{singletary2019online}
A.~Singletary, P.~Nilsson, T.~Gurriet, and A.~D. Ames, ``Online active safety
  for robotic manipulators,'' in \emph{2019 IEEE/RSJ International Conference
  on Intelligent Robots and Systems (IROS)}, 2019, pp. 173--178.

\bibitem{gurriet2020scalable}
T.~Gurriet, M.~Mote, A.~Singletary, P.~Nilsson, E.~Feron, and A.~D. Ames, ``A
  scalable safety critical control framework for nonlinear systems,''
  \emph{IEEE Access}, vol.~8, pp. 187\,249--187\,275, 2020.

\bibitem{chen2021backup}
Y.~Chen, M.~Jankovic, M.~Santillo, and A.~D. Ames, ``Backup control barrier
  functions: Formulation and comparative study,'' in \emph{2021 60th IEEE
  Conference on Decision and Control (CDC)}, 2021, pp. 6835--6841.

\bibitem{squires2018constructive}
E.~Squires, P.~Pierpaoli, and M.~Egerstedt, ``Constructive barrier certificates
  with applications to fixed-wing aircraft collision avoidance,'' in \emph{2018
  IEEE Conference on Control Technology and Applications (CCTA)}, 2018, pp.
  1656--1661.

\bibitem{breeden2021high}
J.~Breeden and D.~Panagou, ``High relative degree control barrier functions
  under input constraints,'' in \emph{2021 60th IEEE Conference on Decision and
  Control (CDC)}, 2021, pp. 6119--6124.

\bibitem{xiao2021adaptive}
W.~Xiao, C.~Belta, and C.~G. Cassandras, ``Adaptive control barrier
  functions,'' \emph{IEEE Transactions on Automatic Control}, vol.~67, no.~5,
  pp. 2267--2281, 2021.

\bibitem{liu2023auxiliary}
S.~Liu, W.~Xiao, and C.~A. Belta, ``Auxiliary-variable adaptive control barrier
  functions for safety critical systems,'' in \emph{2023 62nd IEEE Conference
  on Decision and Control (CDC)}, 2023, pp. 8602--8607.

\bibitem{xiao2022sufficient}
W.~Xiao, C.~A. Belta, and C.~G. Cassandras, ``Sufficient conditions for
  feasibility of optimal control problems using control barrier functions,''
  \emph{Automatica}, vol. 135, p. 109960, 2022.

\bibitem{Khalil:1173048}
\BIBentryALTinterwordspacing
H.~K. Khalil, \emph{Nonlinear systems; 3rd ed.}\hskip 1em plus 0.5em minus
  0.4em\relax Upper Saddle River, NJ: Prentice-Hall, 2002, the book can be
  consulted by contacting: PH-AID: Wallet, Lionel. [Online]. Available:
  \url{https://cds.cern.ch/record/1173048}
\BIBentrySTDinterwordspacing

\bibitem{xu2022feasibility}
K.~Xu, W.~Xiao, and C.~G. Cassandras, ``Feasibility guaranteed traffic merging
  control using control barrier functions,'' in \emph{2022 American Control
  Conference (ACC)}.\hskip 1em plus 0.5em minus 0.4em\relax IEEE, 2022, pp.
  2309--2314.

\bibitem{xiao2022event}
W.~Xiao, C.~Belta, and C.~G. Cassandras, ``Event-triggered control for
  safety-critical systems with unknown dynamics,'' \emph{IEEE Transactions on
  Automatic Control}, 2022.

\bibitem{zeng2021safety}
J.~Zeng, B.~Zhang, Z.~Li, and K.~Sreenath, ``Safety-critical control using
  optimal-decay control barrier function with guaranteed point-wise
  feasibility,'' in \emph{2021 American Control Conference (ACC)}, 2021, pp.
  3856--3863.

\end{thebibliography}
\end{document}